\documentclass{article}

\usepackage{epsfig,graphicx}
\usepackage{indentfirst}
\usepackage{amsmath,amsfonts,amsthm,amssymb}
\usepackage{mathrsfs}
\usepackage{amscd}
\usepackage{slashbox}

\DeclareMathAlphabet{\mathsfsl}{OT1}{cmss}{m}{sl}

\newtheorem{thm}{Theorem}[section]
\newtheorem{lem}[thm]{Lemma}
\newtheorem{cor}[thm]{Corollary}
\newtheorem{prop}[thm]{Proposition}

\theoremstyle{definition}

\newcommand{\HF}{\widehat{HF}}

\newcommand{\Z}{\mathbb{Z}}

\newcommand{\Q}{\mathbb{Q}}

\newcommand{\spinc}{${\rm Spin}^c$ }

\newcommand\CFKi{CFK^\infty}
\newcommand\bbF{\mathbb{F}}

\newcommand{\spin}{\mathrm{Spin}^c}

\begin{document}

\title{A cabling formula for $\nu^+$ invariant}

\author{{\large Zhongtao WU}\\{\normalsize Department of Mathematics, The Chinese Universiy of Hong Kong}\\
{\normalsize Lady Shaw Building, Shatin, Hong Kong}\\{\small\it Emai\/l\/:\quad\rm ztwu@math.cuhk.edu.hk}}

\date{}
\maketitle
\begin{abstract}
We prove a cabling formula for the concordance invariant $\nu^+$, defined by the author and Hom in \cite{HomWu}. This gives rise to a simple and effective $4$-ball genus bound for many cable knots.  

\end{abstract}

\section{Introduction}
The invariant $\nu^+$, or equivalently $\nu^-$, is a concordance invariant defined by the author and Hom \cite{HomWu}, and by Ozsv\'ath-Stipsicz-Szab\'o \cite{OSS} based on Rasmussen's local $h$ invariant \cite{RasThesis}. It gives a lower bound on the $4$-ball genus of knots and can get arbitrarily better than the bounds from Ozsv\'ath-Szab\'o $\tau$ invariant.  In this paper, we prove a cabling formula for $\nu^+$.  The main result is:

\begin{thm}\label{main}
For $p,q>0$ and the cable knot $K_{p,q}$, we have $$\nu^+(K_{p,q})=p\nu^+(K)+\frac{(p-1)(q-1)}{2}$$ when $q\geq (2\nu^+(K)-1)p-1$.
\end{thm}

\noindent
As an application of the cabling formula, one can use $\nu^+$ to bound the $4$-ball genus of cable knots; in certain special cases, $\nu^+$ determines the $4$-ball genus precisely.

\begin{cor}\label{fourball}
Suppose $K$ is a knot such that $\nu^+(K)=g_4(K)=n$.  Then $$\nu^+(K_{p,q})=g_4(K_{p,q})=pn+\frac{(p-1)(q-1)}{2}$$ for $q\geq (2n-1)p-1$.
\end{cor}

\noindent
Take $K=T_{2,5} \# 2 T_{2,3} \# -T_{2,3; 2,5}$, for example.  It is known that $g_4(K)=\nu^+(K)=2$.  Using Corollary \ref{fourball}, we can determine the $4$-ball genus of any cable knots $K_{p,q}$ when $q\geq 3p-1$.  This generalizes \cite[Proposition 3.5]{HomWu}.

Regarding the behavior of $\tau$ invariant under knot cabling, the question was well-studied \cite{HeddencablingII}\cite{Petkova}\cite{VC10}, culminating in Hom's explicit formula in \cite{Homcables}.  In contrast to the rather explicit computational approach used in these papers, our method of study is based on a special relationship between $\nu^+$ and surgery of knots, and thus avoids the potential difficulty associated to the computation of the knot Floer complex $CFK^\infty(K_{p,q})$. 

In order to carry out our proposed method, we need to compute the correction terms on both sides of the reducible surgery (\ref{redeq}), which we shall describe in Section 3.  The most technical part of the argument is to identify the projection map of the \spinc structures in the reducible surgery, and this is discussed in Section 4.   The proof of the main theorem follows in Section 5. 

\vspace{5pt}\noindent{\bf Acknowledgements.} We like to thank Yi Ni for a helpful discussion. The author was partially supported by a grant from the Research Grants Council of the Hong Kong Special Administrative Region, China (Project No. CUHK 2191056).

\section{The invariant $\nu^+$}

In this section, we review the definition and properties of the $\nu^+$ invariant from \cite{HomWu} and relevant backgrounds in Heegaard Floer theory. Heegaard Floer homology is a collection of invariants for closed three-manifolds $Y$ in the form of homology theories $HF^\infty(Y)$, $HF^+(Y)$, $HF^-(Y)$, $\HF(Y)$ and $HF_{\mathrm{red}}(Y)$.  In  Ozsv\'ath-Szab\'o \cite{OSknots} and Rasmussen \cite{RasThesis}, a closely related invariant is defined for null-homologous knots $K \subset Y$, taking the form of an induced filtration on the Heegaard Floer complex of $Y$.  In particular, let $\CFKi(K)$ denote the knot Floer complex of $K \in S^3$.   Consider the quotient complexes
\[ A^+_k = C\{ \max \{i, j-k\}  \geq 0\}  \quad \textup{ and } \quad B^+ = C\{ i \geq 0\} \]
where $i$ and $j$ refer to the two filtrations. The complex $B^+$ is isomorphic to $CF^+(S^3)$. Associated to each $k$, there is a graded, module map
\[ v^+_k: A^+_k \rightarrow B^+ \]
defined by projection and another map
\[ h^+_k: A^+_k \rightarrow B^+ \]
defined by projection to $C\{j \geq k\}$, followed by shifting to $C\{j \geq 0\}$ via the $U$-action, and concluding with a chain homotopy equivalence between $C\{j \geq 0\}$ and $C\{i \geq 0\}$.
Finally, the $\nu^+$ invariant is defined as
\begin{equation}\label{def1}
\nu^+(K):=\mathrm{min} \{k\in \Z \,|\, v^+_k: A^+_k \rightarrow CF^+(S^3), \;\; v^+_k(1)=1 \}.
\end{equation}
Here, $1$ denotes the lowest graded generator of the non-torsion class in the homology of the complex, and we abuse our notations by identifying $A^+_k$ and $CF^+(S^3)$ with their homologies.  

Recall that in the large $N$ surgery, $v^+_k$ corresponds to the maps induced on $HF^+$ by the two handle cobordism from $S^3_N(K)$ to $S^3$ \cite[Theorem 4.4]{OSknots}. This allows one to extract $4$-ball genus bound from functorial properties of the cobordism map. We list below some additional properties of $\nu^+$, all of which can be found in \cite{HomWu}.
\begin{itemize}
\item[(a)]
$\nu^+(K)$ is a smooth concordance invariant, taking nonnegative integer value.
\item[(b)]
$\tau(K)\leq \nu^+(K) \leq g_4(K)$.  (See \cite{OSz4Genus} for the definition of $\tau$)
\item[(c)]
For a quasi-alternating knot $K$, $\nu^+(K)=\left\{
\begin{array}{ll}
0&\text{if } \sigma(K)\ge0,\\
-\frac{\sigma(K)}{2}&\text{if }\sigma(K)<0.
\end{array}
\right.$

\item[(d)]
For a strongly quasi-positive knot $K$,
\[ \nu^+(K) = \tau(K) = g_4(K) = g(K). \]

\end{itemize}

For a rational homology $3$--sphere $Y$ with a Spin$^c$ structure $\mathfrak s$, $HF^+(Y,\mathfrak s)$ is the direct sum of two groups: the first group is the image of $HF^\infty(Y,\mathfrak s)\cong\mathbb \bbF[U,U^{-1}]$ in $HF^+(Y,\mathfrak s)$,
which is isomorphic to $\mathcal T^+=\bbF[U,U^{-1}]/U\bbF[U]$, and its minimal absolute  $\mathbb{Q}$--grading is an invariant of $(Y,\mathfrak s)$, denoted by $d(Y,\mathfrak s)$, the {\it correction term} \cite{OSabsgr}; the second group is the quotient modulo the above image and is denoted by $HF_{\mathrm{red}}(Y,\mathfrak s)$.  Altogether, we have $$HF^+(Y,\mathfrak s)=\mathcal{T}^+\oplus HF_{\mathrm{red}}(Y,\mathfrak s).$$
Using this splitting, we can associate for each integer $k$ and the knot $K$ a non-negative integer $V_k(K)$ that equals the $U$-exponent of $v_k^+$ restricted to\footnote{Again, we abuse the notations by identifying $A_K^+$ with its homology} $\mathcal T^+ \in A_k^+$. This sequence of $\{V_k\}$ is non-increasing, i.e.,  $V_k\geq V_{k+1}$, and stabilizes at 0 for large $k$.  Observe that the minimum $k$ for which $V_k=0$ is the same as $\nu^+(K)$ defined in (\ref{def1}).  This enables us to reinterpret the $\nu^+$ invariant in the following more concise way.
\begin{equation}\label{def2}
\nu^+(K)=\mathrm{min} \{k\in \Z \,|\, V_k=0 \},
\end{equation}

In addition, the sequence $\{V_k\}$ completely determines the correction terms of manifolds obtained from knot surgery.  This can be seen from the surgery formula \cite[Proposition 1.6]{NiWu}.
\begin{equation}\label{Corr}
d(S^3_{p/q}(K),i)=d(L(p,q),i)-2\max\{V_{\lfloor\frac{i}q\rfloor},V_{\lfloor\frac{p+q-1-i}q\rfloor}\}.
\end{equation}
for $p,q>0$ and $0\le i\le p-1$.  We will explain this formula in greater detail in Section 4.

We conclude this section by mentioning that an invariant equivalent to $\nu^+$, denoted $\nu^-$ by Ozsv\'ath and Szab\'o, was formulated in terms of the chain complex $CFK^-$ in \cite{OSS}.  That invariant played an important role to establish a $4$-ball genus bound for a one-parameter concordance invariant $\Upsilon_K(t)$ defined in the same reference.  For our purpose in the rest of the paper, we will not elaborate on that definition. 

\section{Reducible surgery on cable knots}

Recall that the $(p,q)$ cable of a knot $K$, denoted $K_{p,q}$, is a knot supported on the boundary of a tubular neighborhood of $K$ with slope $p/q$ with respect to the standard framing of this torus.  A well-known fact in low-dimensional topology states that the $pq$-surgery on $K_{p,q}$ results in a reducible $3$-manifold.

\begin{prop}\label{redsurg}
\begin{equation}\label{redeq}
S^3_{pq}(K_{p,q})\cong S^3_{q/p}(K) \# L(p,q)
\end{equation}
\end{prop}

The above homeomorphism is exhibited in many references (cf \cite{HeddencablingII}). For self-containedness, we include a proof of Proposition \ref{redsurg} below.  Not only is this reducible surgery a key ingredient of establishing our main result Theorem \ref{main}, the geometric description of the homeomorphism is also crucial for justifying Lemma \ref{spinproj}.

\begin{proof}[Proof of Proposition \ref{redsurg}]
Denote $N(K)$ the tubular neighborhood of $K$ and $E(K)=S^3-N(K)$ its complement, and let $T(K)$ be the boundary torus of $N(K)$.  The cable $K_{p,q}$ is embedded in $T(K)$ as a curve of slope $p/q$.  Consider the tubular neighborhood $N(K_{p,q})$ of the cable.  The solid torus $N(K_{p,q})$ intersects $T_K$ at an annular neighborhood $A=N(K_{p,q})\cap T(K)$, and the boundary of this annulus consists of two parallel copies of $K_{p,q}$, denoted by $\lambda$ and $\lambda '$, each of which have linking number $pq$ with $K_{p,q}$.  Therefore, the surgery slope of coefficient $pq$ is given by $\lambda$ (or equivalently, $\lambda '$), and the $pq$-surgery on $K_{p,q}$ is performed by gluing a solid torus to the knot complement $E(K_{p,q})$ in such a way that the meridian is identified with a curve isotopic to $\lambda$.

\begin{figure}
\includegraphics[scale=0.5]{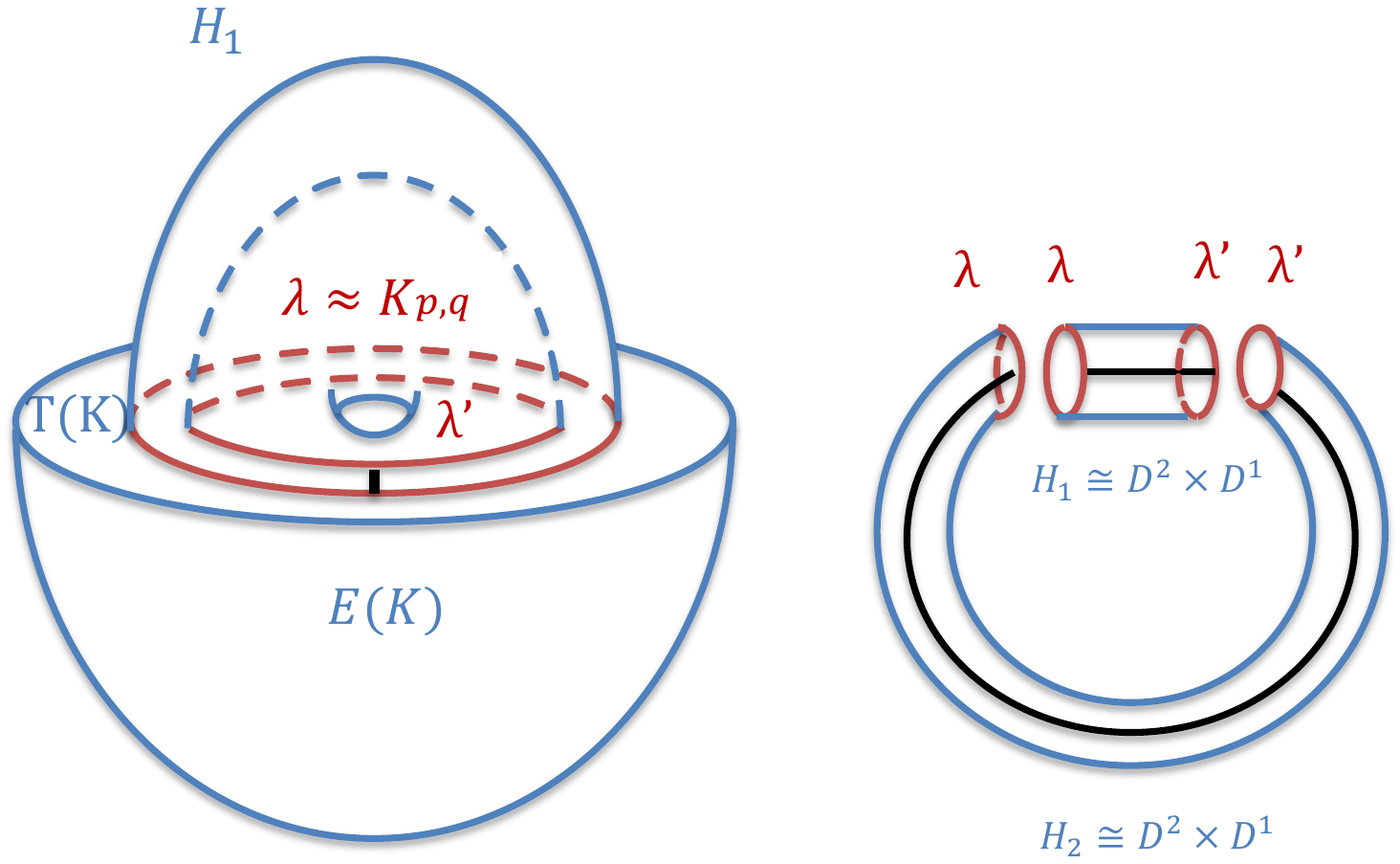}
\caption{Attach the $2$-handle $H_1$ to $E(K)$ along $K_{p,q}$.  The two disk-ends of $H_1$ are identified with the corresponding disk ends of the other $2$-handle $H_2$ that is attached to $N(K)$.     }
\label{handle}
\end{figure}

On the other hand, one can think of the above gluing as attaching a pair of $2$-handles $H_1$, $H_2$ to $E(K_{p,q})$.  See Figure \ref{handle}.  Since the exterior of $K_{p,q}$ is homeomorphic to $$E(K_{p,q})=E(K)\cup_{T(K)-A} N(K),$$ its $pq$-surgery can be decomposed as
$$S^3_{pq}(K_{p,q})=[E(K)\cup H_1]\cup [N(K)\cup H_2].$$ As the $2$-handles are attached along essential curves on $T(K)$ (isotopic to $\lambda$), $E(K)\cup H_1$ and $N(K)\cup H_2$ end up having a common boundary homeomorphic to $S^2$.  This proves that $S^3_{pq}(K_{p,q})$ is a reducible manifold.

To further identify the two pieces of the reducible manifold, note that the attaching curve is isotopic to $\lambda$, which has slope $p/q$ on $T(K)$.  It follows that $$N(K)\cup H_2 \cong L(p,q)-D^3.$$  From the perspective of $E(K)$, the curve $\lambda$ has slope $q/p$.  Thus, the other piece is $$E(K)\cup H_1 \cong S^3_{q/p}(K)-D^3.$$  This completes the proof.

\end{proof}

\section{\spinc structures in reducible surgery}

Let us take a closer look at the surgery formula (\ref{Corr}), in which there is an implicit identification of \spinc structure $$\sigma: \Z/p\Z \rightarrow \spin(S^3_{p/q}(K))$$ For simplicity, we use an integer $0\leq i \leq p-1$ to denote the \spinc structure $\sigma([i])$, when $[i]\in \Z/p\Z$ is the congruence class of $i$ modulo $p$.  The identification can be made explicit by the procedure in  Ozsv\'ath and Szab\'o \cite[Section 4,7]{OSrational}.  In particular, it is independent of the knot $K$ on which the surgery is applied\footnote{Thus, formula (\ref{Corr}) may be interpreted as comparing the correction terms of the ``same'' \spinc structure of surgery on different knots.}; and it is affine:  $$\sigma[i+1]-\sigma[i]=[K']\in H_1(S^3_{p/q}(K))\cong \spin(S^3_{p/q}(K))$$ where $K'$ is the dual knot of the surgery on $K$, and \spinc structures are affinely identified with the first homology.  Moreover, the conjugation map $J$ on \spinc structures can be expressed as
\begin{equation}\label{conjspin}
J(\sigma([i]))=\sigma([p+q-1-i])
\end{equation}
(cf \cite[Lemma 2.2]{LiNi}).  We will use these identifications throughout this paper.

In \cite[Proposition 4.8]{OSabsgr}, Ozsv\'ath and Szab\'o made an identification of \spinc structures on lens spaces through their standard genus 1 Heegaard diagram, which coincide with the above identification through surgery (on the unknot). They also proved the following recursive formula for the correction terms of lens spaces
\begin{equation}\label{lenscor}
d(L(p,q),i)=\frac{(2i+1-p-q)^2-pq}{4pq}-d(L(q,r),j)
\end{equation}
for positive integers $p>q$ and $0\leq i <p+q$, where $r$ and $j$ are the reduction module $q$ of $p$ and $i$, respectively.  Substituting in $q=1$, one sees:
\begin{equation}\label{lenscor2}
d(L(p,1),i)=\frac{(2i-p)^2-p}{4p}
\end{equation}

For a reducible manifold $Y=Y_1\#Y_2$,  there are projections from $\spin (Y)$ to the \spinc structure of the two factors $\spin (Y_1)$ and $\spin(Y_2)$. Particularly, this applies to the case of the reducible surgery $S^3_{pq}(K_{p,q})\cong S^3_{q/p}(K) \# L(p,q)$. In terms of the canonical identification above, we write $\phi_1: \Z/pq\Z \rightarrow \Z/q\Z$ and $\phi_2: \Z/pq\Z \rightarrow \Z/p\Z$ for the two projections. These two maps are independent of the knot $K$, which we determine explicitly in the next lemma.

With the above notations and identifications of \spinc structures understood, we apply the surgery formula (\ref{Corr}) to both sides of the reducible manifold $S^3_{pq}(K_{p,q})\cong S^3_{q/p}(K) \# L(p,q)$ and deduce

\begin{multline}\label{longeq}
d(L(pq, 1),i)-2V_i(K_{p,q})=d(L(q,p), \phi_1(i))+d(L(p,q), \phi_2(i))\\
-2\max\{V_{\lfloor\frac{\phi_1(i)}p\rfloor}(K),V_{\lfloor\frac{p+q-1-\phi_1(i)}p\rfloor}(K)\}
\end{multline}
for all $i\leq \frac{pq}{2}$.  Here we used the fact that $V_i \geq V_{pq-i}$ when $i\leq \frac{pq}{2}$, as $\{V_k\}$ is a non-increasing sequence.  When $K$ is the unknot, (\ref{longeq}) simplifies to
\begin{equation}\label{eq1}
d(L(pq, 1),i)-2V_i(T_{p,q})=d(L(q,p), \phi_1(i))+d(L(p,q), \phi_2(i))
\end{equation}
as all $V_i$'s are 0 for the unknot.

For the rest of the section, assume $Y=S^3_{pq}(K_{p,q})$, $Y_1=S^3_{q/p}(K)$ and $Y_2=L(p,q)$, and denote $K'\subset Y_1=S^3_{q/p}(K)$ and $K_{p,q}'\subset Y= S^3_{pq}(K_{p,q})$ the dual knots of the surgery on $K$ and $K_{p,q}$, respectively.

\begin{lem}\label{spinproj}
The projection maps of the \spinc structure $\phi_1$ and $\phi_2$ are given by:
$$\phi_1(i)=i-\frac{(p-1)(q-1)}{2} \, \pmod{q};$$
$$\phi_2(i)=i-\frac{(p-1)(q-1)}{2} \, \pmod{p}.$$

\end{lem}

\begin{proof}
Since the projection maps are affine, we assume $$\phi_1(i)=a_1\cdot i+b_1 \pmod{q},$$ $$\phi_2(i)=a_2\cdot i+b_2 \pmod{p}.$$  Note that the maps $\phi_1$, $\phi_2$ are generally not homomorphisms.  Nevertheless, we claim that $\phi_1(i+1)-\phi_1(i)=1 \pmod{q}$.  Under Ozsv\'ath-Szab\'o's canonical identification $\sigma: \Z/pq\Z \rightarrow \spin(Y)$, we have $\sigma[i+1]-\sigma[i]=[K_{p,q}']\in H_1(Y)\cong \Z/pq\Z$. Thus, it amounts to show that $\phi_{1}[K_{p,q}']=[K']$ under the projection of $Y$ into the first factor $Y_1$.

This can be seen from the geometric description of the reducible surgery in last section: The dual knot $K_{p,q}'$, isotopic to the closed black curve on the right of Figure \ref{handle}, projects to an arc in $E(K)\cup H_1$ on the left of Figure \ref{handle}.  This arc is closed up in $Y_1$ by connecting it to a simple arc in $D^3=Y_1-(E(K)\cup H_1)$.  Since the curve intersects $\lambda$ once, it must represent\footnote{To be accurate, this is true up to a proper choice of orientation of $K'$.}  $[K']\in H_1(Y_1)$.
Hence
 $$1=\phi_1(i+1)-\phi_1(i)=(a_1\cdot (i+1)+b_1)-(a_1\cdot i +b_1) \pmod{q} $$ from which we see $a_1=1$.  A similar argument proves $a_2=1$.

To determine $b_1$, note that the projection $\phi_1$ commutes with the conjugation $J$ as operations on \spinc structures.
After substituting the equation $\phi_1(i)=i+b_1 \, \pmod{q}$ into $\phi_1\circ J=J\circ \phi_1$ and applying (\ref{conjspin}), we get
$$(pq-i)+b_1 =p+q-1-(i+b_1) \pmod{q}.$$
So
$$b_1=\left\{
	\begin{array}{ll}
	-\frac{(p-1)(q-1)}{2}  &\textup{if } q \textup{ is odd}\\
	-\frac{(p-1)(q-1)}{2} \; \textup{or} \, -\frac{(p-1)(q-1)}{2}+\frac{q}{2} &\textup{if } q \textup{ is even}
	\end{array}
	\right.
$$
where the identity is understood modulo $q$ as before.  Similar arguments also imply:
$$b_2=\left\{
	\begin{array}{ll}
	-\frac{(p-1)(q-1)}{2}  &\textup{if } p \textup{ is odd}\\
	-\frac{(p-1)(q-1)}{2} \; \textup{or} \, -\frac{(p-1)(q-1)}{2}+\frac{p}{2} &\textup{if } p \textup{ is even}
	\end{array}
	\right.
$$

We argue that $b_1=b_2=-\frac{(p-1)(q-1)}{2} $.  This is evidently true when both $p$ and $q$ are odd integers.  When $p$ is even and $q$ is odd, we want to exclude the possibility $b_1=-\frac{(p-1)(q-1)}{2}$ and $b_2=-\frac{(p-1)(q-1)}{2}+\frac{p}{2}$ using the method of proof by contradiction. A similar argument will address the case for which $p$ is odd and $q$ is even, and thus completes the proof.

We derive a contradiction by comparing the correction terms computed in two ways.
From equation (\ref{lenscor}) and (\ref{lenscor2}), we have
\begin{eqnarray*}
 d(L(pq,1),j+\frac{(p-1)(q-1)}{2}) &=& \frac{(2j+1-p-q)^2-pq}{4pq}\\
&=& d(L(q,p),j)+d(L(p,q),j)
\end{eqnarray*}
for $0\leq j <p+q$.

On the other hand,  it follows from (\ref{eq1}) that $$d(L(pq, 1),j+\frac{(p-1)(q-1)}{2})=d(L(q,p), j)+d(L(p,q), j+\frac{p}{2})$$ for $0\leq j <p+q-1$.  Here, we used the fact that $V_i(T_{p,q})=0$ for all $i>\frac{(p-1)(q-1)}{2}$ (since $g_4(T_{p,q})=\frac{(p-1)(q-1)}{2}$) and the assumptions $\phi_1(i)=i-\frac{(p-1)(q-1)}{2}$ and $\phi_2(i)=i-\frac{(p-1)(q-1)}{2}+\frac{p}{2}$.  Comparing the above two identities, we obtain
\begin{equation}\label{contraidentity}
d(L(p,q),j)=d(L(p,q),j+\frac{p}{2})
\end{equation}

Recall from Lee-Lipshitz \cite[Corollary 5.2]{LeeLip} that correction terms of lens spaces also satisfy the identity
\begin{equation*}\label{corLeeLip}
d(L(p,q),j+q)-d(L(p,q),j)=\frac{p-1-2j}{p}
\end{equation*}
for $0\leq j<p$.
It follows $$d(L(p,q),j+\frac{p}{2}+q)-d(L(p,q),j+\frac{p}{2})=\frac{p-1-2(j+\frac{p}{2})}{p}=\frac{-1-2j}{p}$$
Yet, according to (\ref{contraidentity}), $$d(L(p,q),j+\frac{p}{2}+q)-d(L(p,q),j+\frac{p}{2})= d(L(p,q),j+q)-d(L(p,q),j)=\frac{p-1-2j}{p}$$
 We reach a contradiction!  This completes the proof of the lemma.

\end{proof}

As a quick check of Lemma \ref{spinproj}, let us look at the surgery $S^3_{15}(T_{3,5})\cong L(5,3) \# L(3,5)$.  The correction terms of the three lens spaces with \spinc structure $i$ are computed using (\ref{lenscor}) and summarized in Table \ref{table1} below.

\begin{table}[ht!]
  \centering
  \begin{tabular}{|c | c c c c c c c c|}
   \hline
   \backslashbox{lens space}{$i$} & $0$ & $1$ & $2$ & $3$ & $4$ & $5$ & $6$ & $\cdots$ \\
   \hline
   $L(15,1)$& $7/2$ & $77/30$ & $53/30$ & $11/10$ & $17/30$ & $1/6$ & $-1/10$ & $\cdots$\\
   $L(5,3)$ & $2/5$& $0$ & $2/5$ & $-2/5$ & $-2/5$ & $2/5$ & $2/5$ & $\cdots$ \\
   $L(3,5)$ & $1/6$ & $1/6$ & $-1/2$ & $1/6$ & $1/6$ & $-1/2$ & $1/6$ & $\cdots$ \\
   \hline
  \end{tabular}
  \caption{Table of correction terms for lens spaces}
  \label{table1}
  \end{table}

Meanwhile, we compute the projection functions $\phi_1(i)$, $\phi_2(i)$ (according to the formula in Lemma \ref{spinproj}) and $V_i(T_{3,5})$, and summarize the results in Table \ref{table2}. We can then verify identity (\ref{eq1}) using the values of correction term
provided in Table \ref{table1}.  In particular, note that at $i=\frac{(p-1)(q-1)}{2}=4$, there is the column $\phi_1=\phi_2=V=0$, as expected from Lemma \ref{spinproj}, and there is the identity $17/30=2/5+1/6$ that we can read off.

\begin{table}[ht!]
  \centering
  \begin{tabular}{|c | c c c c c c c c|}
   \hline
   \backslashbox{Function}{$i$} & $0$ & $1$ & $2$ & $3$ & $4$ & $5$ & $6$ & $\cdots$ \\
   \hline
   $\phi_1$& $1$ & $2$ & $3$ & $4$ & $0$ & $1$ & $2$ & $\cdots$\\
   $\phi_2$ & $2$& $0$ & $1$ & $2$ & $0$ & $1$ & $2$ & $\cdots$ \\
   $V$ & $2$ & $1$ & $1$ & $1$ & $0$ & $0$ & $0$ & $\cdots$ \\
   \hline
  \end{tabular}
  \caption{projections $\phi_1$, $\phi_2$ are given by Lemma \ref{spinproj}}.
\label{table2}
 \end{table}

\section{Proof of cabling formula}

In this section, we prove Theorem \ref{main}. First, note the following relationship between the sequences $V_i(K_{p,q})$ and $V_i(K)$ if we compare (\ref{longeq}) and (\ref{eq1}).

\begin{lem}\label{cablemma}
Given $p,q>0$ and $i\leq \frac{pq}{2}$, the sequence of non-negative integers $V_i(K_{p,q})$ and $V_i(K)$ satisfy the relation
\begin{equation}\label{cabcorrect}
V_i(K_{p,q})=V_i(T_{p,q})+2\max\{V_{\lfloor\frac{\phi_1(i)}p\rfloor}(K),V_{\lfloor\frac{p+q-1-\phi_1(i)}p\rfloor}(K)\}
\end{equation}
Here, $\phi_1(i)=i-\frac{(p-1)(q-1)}{2}$ as above.
\end{lem}

In order to evaluate $\nu^+(K_{p,q})$ from equation (\ref{def2}), it is enough to determine the minimum $i$ such that $V_i(K_{p,q})=0$.  Since $V_i(T_{p,q})>0$ when $i<\frac{(p-1)(q-1)}{2}$, we only need to consider $i\geq \frac{(p-1)(q-1)}{2}$ by (\ref{cabcorrect}).

\begin{proof}[Proof of Theorem \ref{main}]
When $q\geq (2\nu^+(K)-1)p+1$, we have $p\nu^+(K)+\frac{(p-1)(q-1)}{2}\leq \frac{pq}{2}$, so the condition $i\leq \frac{pq}{2}$ in Lemma \ref{cablemma} is satisfied for all $i$ in the range $\frac{(p-1)(q-1)}{2}\leq i \leq p\nu^+(K)+\frac{(p-1)(q-1)}{2}$.  Equation (\ref{cabcorrect}) simplifies to 
\begin{equation}\label{last}
V_i(K_{p,q})=2 V_{\lfloor\frac{\phi_1(i)}p\rfloor}(K)
\end{equation}
 as $V_i(T_{p,q})=0$ and $\lfloor\frac{\phi_1(i)}p\rfloor \leq \lfloor\frac{p+q-1-\phi_1(i)}p\rfloor$ for $\phi_1(i)=i-\frac{(p-1)(q-1)}{2}$.  As $V_i(K)=0$ if and only if $i\geq \nu^+(K)$, it is easy to see from (\ref{last}) and Lemma \ref{spinproj} that the minimum $i$ such that $V_i(K_{p,q})=0$ is $p\nu^+(K)+\frac{(p-1)(q-1)}{2}$.  Hence, $\nu^+(K)=p\nu^+(K)+\frac{(p-1)(q-1)}{2}$.

\end{proof}

When $q< (2\nu^+(K)-1)p+1$, the cabling formula for $\nu^+(K_{p,q})$ is still unknown.  Nevertheless, the preceding argument gives the following lower bound.

\begin{prop}
For $p,q>0$ and the cable knot $K_{p,q}$, $$\nu^+(K_{p,q})\geq \frac{pq}{2}$$ when $q< (2\nu^+(K)-1)p+1$.

\end{prop}

Theorem \ref{main} are useful when one aims to determine the $4$-ball genus of some cable knots (e.g. Corollary \ref{fourball}).

\begin{proof}[Proof of Corollary \ref{fourball}]
The $4$-ball genus of the cable knot $$g_4(K_{p,q})\leq pg_4(K)+\frac{(p-1)(q-1)}{2}$$ since on can construct a slice surface for $K_{p,q}$ from $p$ parallel copies of a slice surface for $K$ together with $(p-1)q$ half-twisted bands. Using Theorem \ref{main}, we see $$pn+\frac{(p-1)(q-1)}{2} =\nu^+(K_{p,q})\leq g_4(K_{p,q})\leq pn+\frac{(p-1)(q-1)}{2}, $$ from which Corollary \ref{fourball} follows immediately.

\end{proof}

\end{document}